\documentclass[12pt,draft,reqno]{amsart}
\usepackage{amsmath}
\usepackage{amssymb}
\usepackage{amsthm}
\usepackage{bm}
\numberwithin{equation}{section}
\newtheorem{theorem}{Theorem}[section]
\newtheorem{lemma}[theorem]{Lemma}
\newtheorem{proposition}[theorem]{Proposition}
\newtheorem{corollary}[theorem]{Corollary}

\newtheorem{definition}[theorem]{Definition}
\newtheorem{thm}{Theorem}

\theoremstyle{definition}
\newtheorem{remark}[theorem]{Remark}

\newcommand{\ep}{\varepsilon}

\newcommand{\tht}{\theta}

\newcommand{\ld}{\lambda}

\newcommand{\qq}{\qquad}

\newcommand{\R}{\mathbb{R}}
\newcommand{\N}{\mathbb{N}}

\newcommand{\rd}{\mathbb R ^d}

\newcommand{\vc}{{\Vec{c}}}
\newcommand{\vs}{{\Vec{s}}}
\newcommand{\vt}{{\Vec{t}}}
\newcommand{\vsig}{{\Vec{\sigma}}}

\allowdisplaybreaks
\begin{document}
\title[Hurwitz-Lerch type of Euler-Zagier double zeta distributions]{Hurwitz-Lerch zeta and Hurwitz-Lerch type of Euler-Zagier double zeta distributions}
\author[T.~Nakamura]{Takashi Nakamura}
\address[T.~Nakamura]{Department of Liberal Arts, Faculty of Science and Technology, Tokyo University of Science, 2641 Yamazaki, Noda-shi, Chiba-ken, 278-8510, Japan}
\email{nakamuratakashi@rs.tus.ac.jp}
\urladdr{https://sites.google.com/site/takashinakamurazeta/}
\maketitle

\begin{abstract}
In this paper, we give Hurwitz-Lerch zeta distributions with $0 < \sigma \ne 1$ by using the Gamma function. Moreover, we define Hurwitz-Lerch type of Euler-Zagier double zeta distributions not only in the region of absolute convergence but also the outside of the region of absolute convergence. 
\end{abstract}

%%%%%%%%%%%%%%%%%%%%%%
\section{Introduction}
%%%%%%%%%%%%%%%%%%%%%%

%%%%%%%%%%%%%%%%%%%%%%%%%%%
\subsection{Zeta functions}
%%%%%%%%%%%%%%%%%%%%%%%%%%%
First, we introduce the following function. 
\begin{definition}[see {\cite[p.~53, (1)]{Er}}]\label{def:ler}
For $0 < a \le 1$, $s,z \in {\mathbb{C}}$ and $0< |z|\le 1$, the Hurwitz-Lerch zeta function $\Phi(s,a,z)$ is defined by
\begin{equation}
\Phi (s,a,z) := \sum_{n=0}^{\infty} \frac{z^n}{(n+a)^s}, \qquad s := \sigma + {\rm{i}}t, 
\quad \sigma >1 , \quad t \in {\mathbb{R}}.
\end{equation}
\end{definition}
The Riemann zeta function $\zeta (s)$ and the Hurwitz zeta function $\zeta (s,a)$ are expressed as $\Phi (s,1,1)$ and $\Phi (s,a,1)$, respectively. The Dirichlet series of $\Phi(s,a,z)$ converges absolutely in the right half-plane $\sigma >1$ and uniformly in each compact subset of this half-plane. The function $\Phi(s,a,z)$ with $z \ne 1$ is analytically continuable to the whole complex plane. However, $\zeta (s,a)$ is analytic for all $s \in {\mathbb{C}}$ except for a simple pole at $s=1$ with residue $1$. On the other hand, Lerch showed that
$$
\biggl[ \frac{d}{ds} \zeta (s,a) \biggr]_{s=0} = \log \Gamma (a) - \frac{1}{2} \log (2\pi) .
$$
Hence the gamma function $\Gamma (a)$ can be written by the Hurwitz zeta function. 

Next we define the following double zeta function as a double sum and two variable version of $\Phi(s,a,z)$.
\begin{definition}[see {\cite[(1)]{KomoQua}}]\label{def:ler}
For $0 < a \le 1$, $s_1,s_2, z_1, z_2 \in {\mathbb{C}}$ and $0<|z_1|, |z_2| \le 1$, the Hurwitz-Lerch type of Euler-Zagier double zeta function $\Phi_2(s_1,s_2,a,z_1,z_2)$ is defined by
\begin{equation}\label{eq:defezhdz}
\Phi_2(s_1,s_2,a,z_1,z_2) := \sum_{m=0}^\infty \frac{z_1^m}{(m+a)^{s_1}} \sum_{n=1}^\infty \frac{z_2^{n-1}}{(m+n+a)^{s_2}}.
\end{equation}
\end{definition}
Note that the function $\Phi_2 (s_1,s_2,a,z_1,z_2)$ can be continued meromorphically to the whole space ${\mathbb{C}}^2$ by Komori in \cite[Theorem 3.14]{KomoQua} (see also Lemma \ref{lem:ezhd1}). It should be mentioned that Atkinson \cite{Atkinson} obtained an analytic continuation for the function $\zeta_2(s_1,s_2 \,;a) := \Phi_2(s_1,s_2,a,1,1)$ with $a=1$ in order to study the mean square of the Riemann zeta function $\int_0^T |\zeta (1/2+it)|^2dt$ in 1945. 

As a generalization of the double zeta functions above to $\rd$-valued, Aoyama and Nakamura \cite{AN12} defined the following Multidimensional Shintani zeta function $Z_S(\vs)$. Let $d,m,r\in\N$, $\vs\in\mathbb{C}^d$ and $(n_1, \ldots , n_r)\in\mathbb{Z}_{\ge 0}^{r}$.
For $\ld_{lj}, u_j > 0$, $\vc_l \in {\mathbb{R}}^d$, where $1\le j\le r$ and $1\le l\le m$, and 
a function $\tht (n_1, \ldots , n_r)\in{\mathbb{C}}$
satisfying $|\tht (n_1, \ldots , n_r)| = O((n_1+ \cdots +n_r)^{\ep})$, for any $\ep >0$, 
we define a multidimensional Shintani zeta function given by
\begin{equation}
Z_S (\vs) := \sum_{n_1 ,\ldots, n_r =0}^{\infty} 
\frac{\tht (n_1,\ldots , n_r)}{\prod_{l=1}^m (\lambda_{l1}(n_1+u_1) 
+\cdots+\lambda_{lr}(n_r+u_r) )^{\langle \vc_l,\vs \rangle}}.
\label{eq:def2}
\end{equation}
The series defined by $\eqref{eq:def2}$ converges absolutely in the region $\min_{1\le l\le m}$ $\Re\langle \vc_l, \vs\rangle >r/m$. This is a multidimensional case of the Shintani multiple zeta functions, when the coefficient $\tht (n_1,\ldots , n_r)$ in \eqref{eq:def2} is a product of Dirichlet characters, 
considered by Hida \cite{Hida} (see also a survey \cite{Masmz}). 

%%%%%%%%%%%%%%%%%%%%%%%%%%%%%%%
\subsection{Zeta distributions}
%%%%%%%%%%%%%%%%%%%%%%%%%%%%%%%
Let $\mu$ be a distribution (probability measure) on ${\mathbb{R}}^d$, namely, $\int_{{\mathbb{R}}^d} \mu (dy) =1$. For $z \in {\mathbb{R}}^d$, the characteristic function $\widehat{\mu} (z)$ of $\mu$ is defined by $\widehat{\mu} (z) := \int_{{\mathbb{R}}^d} e^{{\rm{i}}\langle z,y\rangle} \mu (dy)$, where $\langle \cdot,\cdot \rangle$ is the inner product on ${\mathbb{R}}^d$. By the definition of characteristic function, we can see that the absolute value of characteristic function is not greater than 1 (see for example \cite[Proposition 2.5]{S99}). 

Recently, a class of distribution generated by the Hurwitz zeta function is introduced and studied by Hu, Iksanov, Lin, and Zakusylo \cite{Hu06}. Put the corresponding normalized function and a discrete one-sided random variable $X_a (\sigma)$ as follows:
\begin{equation}
g_{\sigma,a}(t):=\frac{\zeta (\sigma +{\rm i}t,a)}{\zeta (\sigma,a)}, \qq t\in\R,
\label{eq:defhuzd}
\end{equation}
where $\sigma >1$ and
\begin{equation*}
{\rm Pr} \bigl(X_a (\sigma)=\log (n+a) \bigr)=\frac{(n+a)^{-\sigma}}{\zeta (\sigma,a)},
\qq n\in\N\cup\{0\}.
\end{equation*}
Then $g_{\sigma,a} (t)$ is known to be a characteristic function of $-X_a (\sigma)$ (see {\cite[Theorem 1]{Hu06}). Therefore, we can define the following distribution.
\begin{definition}
A distribution $\mu_{\sigma,a}$ on $\R$ is said to be a Hurwitz zeta distribution with parameter $(\sigma,a)$ 
if it has $g_{\sigma,a} (t)$ as its characteristic function.
\end{definition}

Put $g_{\sigma}(t):=\zeta (\sigma +{\rm i}t)/\zeta (\sigma)$, $t\in\R$, then $g_{\sigma}(t)$ is known to be a characteristic function of the Riemann zeta distribution $\mu_{\sigma}$. Aoyama and Nakamura \cite{AN12} defined  the following Multidimensional Shintani zeta distribution. Let $\tht (n_1, \ldots , n_r)$ be a nonnegative or non-positive definite function and $\vsig$ satisfy $\min_{1\le l\le m}\langle \vc_l, \vsig\rangle >r/m$. Then the multidimensional Shintani zeta random variable $X_{\vsig}$ with probability distribution on $\rd$ given by
\begin{align*}
{\rm Pr} \biggl(X_{\vsig}= 
\biggl(\,& -\sum_{l=1}^m c_{l1} \log \bigl(\lambda_{l1}(n_1+u_1) +\cdots+\lambda_{lr}(n_r+u_r)\bigr),\\
&\dots , -\sum_{l=1}^m c_{ld} \log \bigl(\lambda_{l1}(n_1+u_1) +\cdots+\lambda_{lr}(n_r+u_r) \bigr)
\,\biggr) \biggr)\\
=\ \ \ \ &\frac{\tht (n_1,\ldots , n_r)}{Z_S(\vsig)}
\prod_{l=1}^m \bigl(\lambda_{l1}(n_1+u_1) +\cdots+\lambda_{lr}(n_r+u_r) \bigr)^{-\langle \vc_l,\vsig \rangle}.
\end{align*}
In \cite[Theorem 3]{AN12}, the following is proved. Let $X_{\vsig}$ be a multidimensional Shintani zeta random variable.
Then its characteristic function $g_{\vsig}$ is given by
\begin{align*}
g_{\vsig}(\vt)=\frac{Z_S(\vsig +{\rm i}\vt)}{Z_S(\vsig)}, \qq \vt\in\rd.
\end{align*}
They also showed the Multidimensional Shintani zeta distribution contains some fundamental probability distributions on $\R$ such as binomial and Poisson distributions (see \cite[Example 3]{AN12}). 

%%%%%%%%%%%%%%%%
\subsection{Aim}
%%%%%%%%%%%%%%%%
Ramachandra and Sankaranarayanan showed the following inequality.
\begin{thm}[{\cite[Theorem 1]{RaSa}}]
Let $1/2 \le \sigma_0 <1$, $0\le \theta <2\pi$, $\varepsilon >0$. Let $l$ be an integer constant satisfying $l \ge 6$, $y_0$ be the positive solution of $e^{y_0}=2y_0+1$, $c_2 := 2y_0(2y_0+1)^{-2}$ and $0< c_1 < c_2$. Then for $T \ge T_0$ depending on these constants, we have
$$
\Re \bigl( e^{-{\rm{i}}\theta} \zeta (\sigma_0+{\rm{i}}t, a) \bigr) \ge
\frac{c_0 c_1}{1-\sigma_0} (\log t_0)^{1-\sigma_0}
$$
for at least one $t_0$ in $T^\varepsilon/2 \le t_0 \le 3T/2$, where $c_0:= \cos (2\pi /l) (\log l)^{\sigma_0-1}$.
\end{thm}
Hence for any $0 <a \le 1$ and $1/2 \le \sigma <1$, there exists $t_0 \in {\mathbb{R}}$ such that 
\begin{equation}\label{ieq:hur0}
|\zeta (\sigma+{\rm{i}}t_0, a)| > |\zeta (\sigma, a)| .
\end{equation}
Therefore, $g_{\sigma,a}(t)$ defined by (\ref{eq:defhuzd}) is not a characteristic function when $1/2 \le \sigma <1$ by (\ref{ieq:hur0}) and the fact that the absolute value of characteristic function is not greater than 1. Despite of this fact, Nakamura \cite{NaPa} proved that the following function
$$
\frac{\sigma}{\zeta (\sigma)} \frac{\zeta (\sigma - {\rm{i}}t)}{\sigma - {\rm{i}}t} =
\frac{\sigma}{\sigma - {\rm{i}}t} \frac{\zeta (\sigma - {\rm{i}}t)}{\zeta (\sigma)}
$$
is a characteristic function for any $0 < \sigma \ne 1$. It should be mentioned that $\sigma(\sigma-{\rm{i}}t)^{-1}$ is the characteristic function of the exponential distribution with parameter $\sigma>0$ defined by $\mu (B) := \sigma \int_{B \cap (0,\infty)} e^{-\sigma y}dy$, where $B$ is a Borel set on ${\mathbb{R}}$ (see for instance \cite[Example 2.14]{S99}). As an application to analytic number theory, Nakamura \cite{NaPa} showed that for any $C \in {\mathbb{C}}$ satisfying $|C| > 10$ and $-19/2 \le \Re (C) \le 17/2$, the function $\zeta(s) + Cs$ does not vanish in the half-plane $1/18 < \sigma$. 

In the present paper, we give Hurwitz-Lerch zeta distributions with $0 < \sigma \ne 1$ by using not $\sigma(\sigma - {\rm{i}}t)^{-1}$ but the normalized Gamma function $\Gamma (\sigma + {\rm{i}}t)/\Gamma (\sigma)$ (see Proposition \ref{th:1} and Theorem \ref{th:2}). Recall that $g_{\sigma,a}(t)$ defined by (\ref{eq:defhuzd}) is not a characteristic function when $1/2 \le \sigma <1$. Next we define Hurwitz-Lerch type of Euler-Zagier double zeta distributions not only in the region of absolute convergence $\Re (s_1) >0$, $\Re (s_2) >1$ and $\Re (s_1+s_2)>2$ but also outside of the region above, for example, $\Re (s_1)>0$, $\Re (s_2) >1$ and $1< \Re (s_1+s_2) <2$ in Theorems \ref{th:dz1} and \ref{th:dz2}. It should be emphasized that this is the first two dimensional zeta distribution which can be continued to the outside of the region of absolute convergence. 

%%%%%%%%%%%%%%%%%%%%%%
\section{Main results}
%%%%%%%%%%%%%%%%%%%%%%

%%%%%%%%%%%%%%%%%%%%%%%%%%%%%%%%%%%%%%%%%%%%%
\subsection{Hurwitz-Lerch zeta distributions}
%%%%%%%%%%%%%%%%%%%%%%%%%%%%%%%%%%%%%%%%%%%%%

For $\sigma >0$, let 
\begin{equation}\label{eq:def1}
F_{\sigma, a, z} (t) := \frac{f_{\sigma, a, z} (t)}{f_{\sigma, a, z} (0)}, \qquad 
f_{\sigma, a, z} (t) := \Gamma (\sigma + {\rm{i}}t) \Phi (\sigma + {\rm{i}}t, a,z).
\end{equation}
Then we have the following statements.
\begin{proposition}\label{th:1}
The function $F_{\sigma, a, z} (t)$ is a characteristic function of a probability measure for all $\sigma >1$ if and only if $0 \ne z \in [-1,1]$. Moreover, the associated probability measure is absolutely continuous with density function $P_{\sigma, a,z} (y)$ is given as follows:
\begin{align}\label{eq:pdef1}
P_{\sigma, a,z} (y) := 
\frac{e^{y\sigma} \exp ((1-a)e^{y})}{f_{\sigma, a,z} (0) (\exp (e^{y})-z)}, \qquad y \in {\mathbb{R}}.
\end{align}
\end{proposition}

\begin{theorem}\label{th:2}We have the following.\\
$(1)$. Let $z=1$. Then the function $F_{\sigma, a, 1} (t)$ is a characteristic function of a probability measure for all $0 < \sigma <1$ if and only if $a \ge 1/2$. Furthermore, the associated probability measure is absolutely continuous with density functions $P_{\sigma, a, 1} (y)$ is given by
\begin{equation}\label{eq:pdef2}
P_{\sigma, a, 1} (y) := \frac{e^{y\sigma}}{f_{\sigma, a, 1} (0)}
\biggl( \frac{\exp ( (1-a) e^{y})}{(\exp (e^{y})-1)} - \frac{1}{e^{y}} \biggr),  \qquad y \in {\mathbb{R}}.
\end{equation}
$(2)$. Let $z \ne 1$. Then the function $F_{\sigma, a, z} (t)$ is a characteristic function of a probability measure for all $\sigma >0$ if and only if $0 \ne z \in [-1,1)$. The associated probability measure is absolutely continuous with density function $P_{\sigma, a,z} (y)$ is written as (\ref{eq:pdef1}).
\end{theorem}

It is widely know that the absolute value of characteristic function is not greater than 1. Therefore, we immediately obtain the following inequality (see also (\ref{inep:phiposiab}) and Lemmas \ref{lem:huzero28} and \ref{lem:phi1posi}).
\begin{corollary}\label{cor1}
Let $a\ge 1/2$. Then for all $t \in {\mathbb{R}}$ and $0 < \sigma \ne 1$, one has
$$
\bigl|\Gamma (\sigma + {\rm{i}}t) \zeta (\sigma + {\rm{i}}t, a)\bigr| \le 
\Gamma (\sigma) \bigl|\zeta (\sigma, a) \bigr|. 
$$
If $0 \ne z \in [-1,1)$, for all $t \in {\mathbb{R}}$ and $\sigma >0$, it holds that
$$
\bigl|\Gamma (\sigma + {\rm{i}}t) \Phi (\sigma + {\rm{i}}t, a,z)\bigr| \le \Gamma (\sigma) \Phi (\sigma, a,z) . 
$$
\end{corollary}

%%%%%%%%%%%%%%%%%%%%%%%%%%%%%%%%%%%%%%%%%%%%%%%%%%%%%%%%%%%%%%%%%%%%%%%%%
\subsection{Hurwitz-Lerch type of Euler-Zagier double zeta distributions}
%%%%%%%%%%%%%%%%%%%%%%%%%%%%%%%%%%%%%%%%%%%%%%%%%%%%%%%%%%%%%%%%%%%%%%%%%
For $\sigma_1 >0$ and $\sigma_2 >0$, let 
\begin{equation}\label{eq:defdzd1}
F_{\vec{\sigma}, a, \vec{z}} (\vec{t}) := 
\frac{f_{\vec{\sigma}, a, \vec{z}} (\vec{t})}{f_{\vec{\sigma}, a, \vec{z}} (\vec{0})}, \quad
f_{\vec{\sigma}, a, \vec{z}} (\vec{t}) := \Gamma (\sigma_1 + {\rm{i}}t_1) \Gamma (\sigma_2 + {\rm{i}}t_2) 
\Phi_2 (\sigma_1 + {\rm{i}}t_1, \sigma_2 + {\rm{i}}t_2, a, z_1,z_2).
\end{equation}
Note that the series expression of $\Phi_2 (s_1, s_2, a, z_1, z_2)$ converges absolutely when $\Re (s_1) >0$, $\Re (s_2) >1$ and $\Re (s_1+s_2)>2$ and the function $\Phi_2 (s_1, s_2, a, z_1, z_2)$ is continued analytically to the outside of the region of absolute convergent (see \cite[Lemma 2.4]{NaHuLe} and \cite[Lemma 2.8]{NaHuLe}, respectively). Then we have the following statements.
\begin{theorem}\label{th:dz1}
The function $F_{\vec{\sigma}, a, \vec{z}} (\vec{t})$ is a characteristic function of a probability measure for all $\sigma_1 >0$, $\sigma_2 >1$ and $\sigma_1+\sigma_2>2$ if and only if $0 \ne z_1,z_2 \in [-1,1]$. Moreover, the associated probability measure is absolutely continuous with density function $P_{\vec{\sigma}, a, \vec{z}} (\eta,\theta)$ is written by 
\begin{align}\label{eq:qdefezh1}
P_{\vec{\sigma}, a, \vec{z}} (\eta,\theta) = 
\frac{e^{\sigma_1\eta} e^{\sigma_2\theta} \exp((1-a)(e^\eta + e^\theta))}{f_{\vec{\sigma},a,\vec{z}} (\vec{0}) 
(\exp(e^\theta)-z_2) (\exp(e^\eta + e^\theta)-z_1)}, \qquad \eta,\theta \in {\mathbb{R}}.
\end{align}
\end{theorem}

\begin{theorem}\label{th:dz2}We have the following.\\
$(1)$. Let $z_1=z_2=1$. Then the function $F_{\vec{\sigma}, a,\vec{1}} (\vec{t})$ is a characteristic function of a probability measure for all $0< \sigma_1 < 1$, $\sigma_2 >1$ and $1< \sigma_1+\sigma_2 <2$ if and only if $a \ge 1/2$. Furthermore, the associated probability measure is absolutely continuous with density function $P_{\vec{\sigma}, a, \vec{1}} (y)$ is given as follows:
\begin{align}\label{eq:qdef2}
P_{\vec{\sigma}, a, \vec{1}} (\eta,\theta) = 
\frac{e^{\sigma_1 \eta} e^{\sigma_2 \theta}}{f_{\vec{\sigma}, a, \vec{1}} (\vec{0})} \left(
\frac{H(a,e^\eta+e^\theta)}{\exp(e^\theta)-1} + \frac{H(1,e^\theta)}{e^\eta + e^\theta} \right), 
\qquad \eta,\theta \in {\mathbb{R}},
\end{align}
where $H(a,x)$ is defined 
\begin{equation}
\label{eq:defHax}
H(a,x) := \frac{e^{(1-a)x}}{e^x-1} - \frac{1}{x} =  \frac{xe^{(1-a)x} - e^x +1}{x(e^x-1)}, \qquad x>0.
\end{equation}
$(2)$. Let $z_1=1$ and $z_2 \ne 1$. Then the function $F_{\vec{\sigma}, a,\vec{z}} (\vec{t})$ is a characteristic function of a probability measure for all $\sigma_1 > 1$ and $\sigma_2 >0$ if and only if $0 \ne z_2 \in [-1,1)$. \\
$(3)$. Let $z_1 \ne 1$ and $z_2=1$. Then the function $F_{\vec{\sigma}, a,\vec{z}} (\vec{t})$ is a characteristic function of a probability measure for all $\sigma_1 > 0$ and $\sigma_2 >1$ if and only if $0 \ne z_1 \in [-1,1)$. \\
$(4)$. Let $z_1 \ne 1$ and $z_2 \ne 1$. Then the function $F_{\vec{\sigma}, a,\vec{z}} (\vec{t})$ is a characteristic function of a probability measure for all $\sigma_1 > 0$ and $\sigma_2 > 0$ if and only if $0 \ne z_1,z_2 \in [-1,1)$. \\
For all cases (2), (3) and (4), the associated probability measure is absolutely continuous with density function $P_{\vec{\sigma}, a, \vec{z}} (\eta,\theta)$ is expressed as (\ref{eq:qdefezh1}). 
\end{theorem}

We obtain the following corollary from the fact that the absolute value of characteristic function is not greater than 1 (see also Lemmas \ref{lem:posi} and \ref{lem:HuEZnega}). 

\begin{corollary}\label{dzcor1}
Suppose $a \ge 1/2$, $\sigma_1 >0$, $\sigma_2 >1$ and $\sigma_1+\sigma_2>2$, or $0< \sigma_1 < 1$, $\sigma_2 >1$ and $1< \sigma_1+\sigma_2 <2$. Then for all $t_1,t_2 \in {\mathbb{R}}$, it holds that
$$
\bigl|\Gamma (\sigma_1 + {\rm{i}}t_1) \Gamma (\sigma_2 + {\rm{i}}t_2) 
\zeta_2 (\sigma_1 + {\rm{i}}t_1, \sigma_2 + {\rm{i}}t_2 \, ; a)\bigr| \le 
\Gamma (\sigma_1) \Gamma (\sigma_2) \bigl| \zeta_2 (\sigma_1, \sigma_2 \, ; a) \bigr| . 
$$
Moreover, for all $t_1,t_2 \in {\mathbb{R}}$, one has
$$
\bigl|\Gamma (\sigma_1 + {\rm{i}}t_1) \Gamma (\sigma_2 + {\rm{i}}t_2) 
\Phi_2 (\sigma_1 + {\rm{i}}t_1, \sigma_2 + {\rm{i}}t_2, a, z_1, z_2)\bigr| \le 
\Gamma (\sigma_1) \Gamma (\sigma_2) \Phi_2 (\sigma_1, \sigma_2, a, z_1, z_2)  
$$
$$
\mbox{for} \quad
\begin{cases}
\Re (s_1) >1 \mbox{ and } \Re (s_2) > 0, & z_1=1 \mbox{ and } 0 \ne z_2 \in [-1,1),\\
\Re (s_1) >0 \mbox{ and } \Re (s_2) > 1, & 0 \ne z_1 \in [-1,1) \mbox{ and } z_2 =1, \\
\Re (s_1) >0 \mbox{ and } \Re (s_2) > 0, & 0 \ne z_1, z_2 \in [-1,1).
\end{cases}
$$
\end{corollary}

%%%%%%%%%%%%%%%%
\section{Proofs}
%%%%%%%%%%%%%%%%

%%%%%%%%%%%%%%%%%%%%%%%%%%%%%%%%%%%%%%%%%%%%%%%%%%%%%%%%%%%%%%%%%%%%
\subsection{Proofs of Proposition \ref{th:1} and Theorem \ref{th:2}}
%%%%%%%%%%%%%%%%%%%%%%%%%%%%%%%%%%%%%%%%%%%%%%%%%%%%%%%%%%%%%%%%%%%%
In order to prove Proposition \ref{th:1}, we quote the following integral representation of $\Gamma (s) \Phi (s,a,z)$.
\begin{lemma}[{see \cite[p.~53, (3)]{Er}}]\label{lem:12.2}
When $z=1$ and $\Re (s) >1$, or $z \ne 1$ and $\Re (s) >0$, 
\begin{equation}\label{eq:gamhur0}
\Phi (s,a,z) = \frac{1}{\Gamma (s)} \int_0^\infty \frac{x^{s-1} e^{(1-a)x}}{e^x-z} dx. 
\end{equation}
\end{lemma}

\begin{proof}[Proof of Proposition \ref{th:1}]
By the change of variables integration $x=e^{y}$ in (\ref{eq:gamhur0}),
\begin{equation*}
\begin{split}
&F_{\sigma, a, z} (t) = \int_{-\infty}^{\infty} 
\frac{e^{y(\sigma-1+{\rm{i}}t)} \exp ( (1-a) e^{y})}{f_{\sigma, a, z} (0) (\exp (e^{y})-z)} e^{y}dy = 
\int_{-\infty}^{\infty} 
e^{{\rm{i}}ty} \frac{e^{y\sigma} \exp ((1-a)e^{y})}{f_{\sigma, a, z} (0) (\exp (e^{y})-z)} dy.
\end{split}
\end{equation*}
First assume that $0 \ne z \in [-1,1]$. Then we obviously have $e^{y\sigma} \exp ((1-a)e^{y})/(\exp (e^{y})-z) >0$ for any $y \in {\mathbb{R}}$. On the other hand, it is well-known that $\Gamma (\sigma)>0$ for all $\sigma >0$. Furthermore, one has
\begin{equation}\label{inep:phiposiab}
\Phi (\sigma,a,z) = \sum_{n=0}^{\infty} \frac{z^n}{(n+a)^{\sigma}} = 
\sum_{n=0}^{\infty} \biggl( \frac{z^{2n}}{(2n+a)^{\sigma}} -\frac{z^{2n+1}}{(2n+1+a)^{\sigma}} \biggr) >0
\end{equation}
from $z^{2n} > |z^{2n+1}| >0$ and $(2n+1+a)^{\sigma} > (2n+a)^{\sigma} >0$. Hence we have $f_{\sigma, a, z} (0) >0$. By the definition of $P_{\sigma, a, z} (y)$ given by (\ref{eq:pdef1}), one has $\int_{-\infty}^{\infty}P_{\sigma, a, z} (y) dy=1$ for any $\sigma >1$. Thus $P_{\sigma, a, z} (y)$ is a probabilistic density function when $\sigma >1$. 

Next suppose $z \not \in [-1,1]$. Then there exist real numbers $y_1$ and $y_2$ such that at least one of $f_{\sigma, a,z} (0) (\exp (e^{y_1})-z)$ and $f_{\sigma, a,z} (0) (\exp (e^{y_2})-z)$ are not real numbers. Hence
$$
\frac{e^{y\sigma} \exp ((1-a)e^{y})}{f_{\sigma, a,z} (0) (\exp (e^{y})-z)} dy
$$
is not a measure but complex signed measure. On the other hand, we have
$$
\bigl| f_{\sigma, a, z} (0) F_{\sigma, a, z} (t) \bigr| \le 
\int_{-\infty}^{\infty}  \frac{e^{y\sigma} \exp ((1-a)e^{y})}{|\exp (e^{y})-z|} dy \le
\int_0^\infty \frac{x^{\sigma-1} e^{(1-a)x}}{e^x-1} dx < \infty.
$$
Therefore, $F_{\sigma, a, z}$ is not a characteristic function when $z \not \in [-1,1]$ since any complex signed measure with finite total variation is uniquely determined by the Fourier transform. 
\end{proof}

We quote following lemmas to show (1) of Theorem \ref{th:2}.
\begin{lemma}[{see \cite[Lemma 2.1]{NaHuLe}}]
\label{lem:12.2ac}
For $0 < \sigma <1$ we have the integral representation
\begin{equation}
\label{eq:gamhurac}
\Gamma (s) \zeta (s,a) =  \int_0^\infty \biggl( \frac{e^{(1-a)x}}{e^x-1} - \frac{1}{x} \biggr) x^{s-1} dx = 
\int_0^\infty \!\!\! H(a,x) x^{s-1} dx.
\end{equation}
\end{lemma}
\begin{lemma}[{see \cite[Lemma 2.2]{NaHuLe}}]
\label{lem:negdefi}
The function $H(a,x)$ defined by (\ref{eq:defHax}) is negative for all $x>0$ if and only if $a \ge 1/2$. 
\end{lemma}
\begin{lemma}[{see \cite[(2.8)]{NaHuLe}}]
\label{lem:huzero28}
When $a \ge 1/2$, one has $\zeta (\sigma,a)<0$ for any $0< \sigma <1$. 
\end{lemma}

\begin{proof}[Proof of (1) of Theorem \ref{th:2}]
By the change of variables integration $x=e^{y}$ in (\ref{eq:gamhurac}),
\begin{equation*}
\begin{split}
F_{\sigma, a,1} (t) = & \int_{-\infty}^{\infty} \frac{e^{y(\sigma-1+{\rm{i}}t)}}{f_{\sigma, a,1} (0)}
\biggl( \frac{\exp ( (1-a) e^{y})}{(\exp (e^{y})-1)} - \frac{1}{e^{y}} \biggr) e^{y}dy \\ = &
\int_{-\infty}^{\infty} e^{{\rm{i}}ty} \frac{e^{y\sigma}}{f_{\sigma, a,1} (0)}
\biggl( \frac{\exp ( (1-a) e^{y})}{(\exp (e^{y})-1)} - \frac{1}{e^{y}} \biggr) dy.
\end{split}
\end{equation*}

Suppose $a \ge 1/2$. Then we have $\exp ((1-a)e^{y})/(\exp (e^{y})-1) - e^{-y} < 0$ for any $y \in {\mathbb{R}}$ by Lemma \ref{lem:negdefi}. Moreover, it holds that $\zeta (\sigma,a)<0$ for all $0< \sigma <1$ from Lemma \ref{lem:huzero28}. By the definition of $P_{\sigma, a,1} (y)$ written as (\ref{eq:pdef2}), one has $\int_{-\infty}^{\infty}P_{\sigma, a,1} (y) dy=1$ for any $0< \sigma <1$. Thus $P_{\sigma, a,1} (y)$ is a probability density function when $a \ge 1/2$. 

Next suppose $0 <a <1/2$. Then there exist $y_1, y_2 \in {\mathbb{R}}$ such that 
$$
\frac{\exp ( (1-a) e^{y_1})}{(\exp (e^{y_1})-1)} - \frac{1}{e^{y_1}} >0,  \qquad
\frac{\exp ( (1-a) e^{y_2})}{(\exp (e^{y_2})-1)} - \frac{1}{e^{y_2}} < 0
$$
from Lemma \ref{lem:negdefi}. Hence 
\begin{equation}\label{mea:1}
\frac{e^{y\sigma}}{f_{\sigma, a} (0)}
\biggl( \frac{\exp ( (1-a) e^{y})}{(\exp (e^{y})-1)} - \frac{1}{e^{y}} \biggr) dy
\end{equation}
is not a measure but signed measure. From the view of \cite[(2.5) and (2.6)]{NaHuLe}, one has
\begin{equation*}
\begin{split}
&\int_{-\infty}^{0} \biggl| 
\frac{\exp ( (1-a) e^{y})}{(\exp (e^{y})-1)} - \frac{1}{e^{y}} \biggr| e^{y\sigma} dy = 
\int_0^1 \biggl| \frac{e^{(1-a)x}}{e^x-1} - \frac{1}{x} \biggr| x^{\sigma-1} dx < \infty ,\\
&\int_{0}^\infty \biggl| \frac{\exp ( (1-a) e^{y})}{(\exp (e^{y})-1)} - \frac{1}{e^{y}} \biggr| e^{y\sigma} dy =
\int_{1}^\infty \biggl| \frac{e^{(1-a)x}}{e^x-1} - \frac{1}{x} \biggr| x^{\sigma-1} dx < \infty.
\end{split}
\end{equation*}
Hence the signed measure (\ref{mea:1}) has finite total variation. It is widely known that any signed measure with finite total variation is uniquely determined by the Fourier transform. Therefore, $F_{\sigma, a,1} (t)$ is not a characteristic function when $0 <a <1/2$. 
\end{proof}

We use the following fact to prove (2) of of Theorem \ref{th:2}.
\begin{lemma}[see {\cite[(2.11)]{NaHuLe}}]\label{lem:phi1posi}
For any $\sigma >0$, $0 < a \le 1$ and $z \in [-1,1)$,  we have
\begin{equation}\label{ineq:hlzposi1}
\Phi (\sigma,a,z)  >0. 
\end{equation}
\end{lemma}
\begin{proof}[Proof of (2) of Theorem \ref{th:2}]
We have to mention that the integral representation (\ref{eq:gamhur0}) converges absolutely for $\sigma >0$ when $z \ne 1$ from $e^x-z \ne 0$ for any $x \ge 0$ and
\begin{equation}\label{in:pfabcon}
\begin{split}
&|\Phi (s,a,z) \Gamma (s)| \le 
\int_0^1 \frac{x^{\sigma-1} e^{(1-a)x}}{|e^x-z|} dx + \int_1^\infty \frac{x^{\sigma-1} e^{(1-a)x}}{|e^x-z|} dx
\\ \le &
\int_0^1 \frac{x^{\sigma-1} e^{(1-a)x}}{|1-z|} dx + \int_1^\infty \frac{x^{\sigma-1} e^{(1-a)x}}{e^x-1} dx
< \infty .
\end{split}
\end{equation}
Hence we can show (2) of Theorem \ref{th:2} by using (\ref{ineq:hlzposi1}), (\ref{in:pfabcon}) and the method used in the proof of Proposition \ref{th:1}. 
\end{proof}

%%%%%%%%%%%%%%%%%%%%%%%%%%%%%%%%%%%%%%%%%%%%%%%%%%%%%%%%%%%%
\subsection{Proofs of Theorems \ref{th:dz1} and \ref{th:dz2}}
%%%%%%%%%%%%%%%%%%%%%%%%%%%%%%%%%%%%%%%%%%%%%%%%%%%%%%%%%%%%
We quote the following fact of the function $\Phi_2 (s_1,s_2, a, z_1, z_2)$ to show Lemma \ref{lem:posi}.
\begin{lemma}[{see \cite[Lemma 2.8]{NaHuLe}}]
\label{lem:ezhd1}
For $\Re (s_1) >1$ and $\Re (s_2) >1$, we have the integral representation
\begin{equation}
\label{eq:gamezhdz2}
\Gamma (s_1) \Gamma (s_2) \Phi_2 (s_1,s_2, a, z_1,z_2) = 
\int_0^\infty \frac{y^{s_2-1}}{e^y-z_2} \int_0^\infty \frac{x^{s_1-1}e^{(1-a)(x+y)}}{e^{x+y}-z_1} dx dy.
\end{equation}
Furthermore, we have the following:\\
${\rm{(1)}}$. When $z_1=z_2=1$, the integral formula (\ref{eq:gamezhdz2}) holds for $\Re (s_1) > 0$, $\Re(s_2)>1$ and $\Re (s_1+s_2) >2$. \\
${\rm{(2)}}$. When $z_1=1$ and $z_2 \ne 1$, the formula (\ref{eq:gamezhdz2}) holds for $\Re (s_1) > 1$ and $\Re(s_2)>0$. \\
${\rm{(3)}}$. When $z_1 \ne 1$ and $z_2 = 1$, the formula (\ref{eq:gamezhdz2}) holds for $\Re (s_1) > 0$ and $\Re(s_2)>1$. \\
${\rm{(4)}}$. When $z_1 \ne 1$ and $z_2 \ne 1$, the formula (\ref{eq:gamezhdz2}) holds for $\Re (s_1) > 0$ and $\Re(s_2)>0$. 
\end{lemma}

From the lemma above, we have the next properties of $\Phi_2 (\sigma_1,\sigma_2, a, z_1, z_2)$.
\begin{lemma}\label{lem:posi}
One has the following:\\
${\rm{(1)}}$. Let $z_1=z_2=1$. Then $\Phi_2 (\sigma_1,\sigma_2, a, 1, 1) >0$ for $\sigma_1 > 0$, $\sigma_2>1$ and $\sigma_1+\sigma_2 >2$. \\
${\rm{(2)}}$. Let $z_1=1$ and $0 \ne z_2 \in [-1, 1)$. Then $\Phi_2 (\sigma_1,\sigma_2, a, 1, z_2) >0$ for $\sigma_1 > 1$, $\sigma_2>0$. \\
${\rm{(3)}}$. Let $0 \ne z_1 \in  [-1,1)$ and $z_2 = 1$. Then $\Phi_2 (\sigma_1,\sigma_2, a, z_1, 1) >0$ for $\sigma_1 > 0$, $\sigma_2>1$. \\
${\rm{(4)}}$. Let $0 \ne z_1,z_2 \in [-1,1)$. Then $\Phi_2 (\sigma_1,\sigma_2, a, z_1, z_2) >0$ for $\sigma_1 > 0$, $\sigma_2>0$. 
\end{lemma}

\begin{proof}
We immediately obtain (1) by the series expression (\ref{eq:defdzd1}). Consider the case (2). The integral representation (\ref{eq:gamezhdz2}) converges absolutely when $\Re (s_1) > 1$ and $\Re(s_2)>0$ since $e^y-z_2 \ne 0$ for any $y \ge 0$ (see (\ref{in:pfabcon})). Hence we have 
\begin{equation}
\label{eq:incase2}
\Gamma (\sigma_1) \Gamma (\sigma_2) \Phi_2 (\sigma_1,\sigma_2, a, 1,z_2) = 
\int_0^\infty \int_0^\infty \frac{y^{\sigma_2-1}}{e^y-z_2} \frac{x^{\sigma_1-1}e^{(1-a)(x+y)}}{e^{x+y}-1} dx dy
\end{equation}
when $\sigma_1 > 1$ and $\sigma_2>0$ from Fubini's theorem. Obviously, we have $e^{(1-a)(x+y)}>0$, $e^y-z_2 >0$ and $e^{x+y}-1 \ge 0$ for any $x,y \ge 0$. Therefore we have (2). We can prove (3) and (4), similarly. 
\end{proof}

\begin{proof}[Proof of Theorem \ref{th:dz1}]
By the change of variables integration $x=e^\eta$ and $y=e^\theta$ in (\ref{eq:gamezhdz2}),
\begin{equation*}
\begin{split}
F_{\vec{\sigma}, a, \vec{z}} (\vec{t}) =&  \int_{-\infty}^{\infty} \int_{-\infty}^{\infty}
\frac{e^{\theta(\sigma_2+{\rm{i}}t_2-1)}}{\exp(e^\theta)-1}  
\frac{e^{\eta(\sigma_1+{\rm{i}}t_1-z_2)} \exp((1-a)(e^\eta + e^\theta))}{f_{\vec{\sigma}, a, \vec{z}} (\vec{0})(\exp(e^\eta + e^\theta)-z_1)} e^\eta e^\theta d\eta d\theta \\ =&
\int_{-\infty}^{\infty} \int_{-\infty}^{\infty} e^{{\rm{i}}( t_1\eta + t_2\theta)}
\frac{e^{\sigma_1 \eta} e^{\sigma_2 \theta}  \exp((1-a)(e^\eta + e^\theta))}{f_{\vec{\sigma}, a, \vec{z}} (\vec{0}) (\exp(e^\theta)-z_2) (\exp(e^\eta + e^\theta)-z_1)} d\eta d\theta .
\end{split}
\end{equation*}
First assume that $0 \ne z_1,z_2 \in [-1,1]$. Then we immediately see that 
$$
\frac{e^{\sigma_1 \eta} e^{\sigma_2 \theta}  \exp((1-a)(e^\eta + e^\theta))}
{(\exp(e^\theta)-z_2) (\exp(e^\eta + e^\theta)-z_1)} >0
$$
for any $\eta,\theta \in {\mathbb{R}}$. On the other hand, $\Phi_2 (\sigma_1,\sigma_2, a, z_1,z_2) >0$ for all $\sigma_1>0$, $\sigma_2 >1$ and $\sigma_1+\sigma_2>2$ by Lemma \ref{lem:posi}. Hence we have $f_{\vec{\sigma}, a, \vec{z}} (\vec{0}) >0$. From the definition of $P_{\vec{\sigma}, a, \vec{z}} (\eta,\theta)$ given by (\ref{eq:qdefezh1}), one has $\int_{\mathbb{R}^2} P_{\vec{\sigma}, a, \vec{z}} (\eta,\theta) d\eta d\theta=1$ for any $\sigma_1>0$, $\sigma_2 >1$ and $\sigma_1+\sigma_2>2$. Thus $P_{\vec{\sigma}, a, \vec{z}} (\eta,\theta)$ is a probabilistic density function when $\sigma_1>0$, $\sigma_2 >1$ and $\sigma_1+\sigma_2>2$. 

Next suppose that $z_1 \not \in [-1,1]$ or $z_2 \not \in [-1,1]$. Then there exist $\eta_0, \eta_2, \theta_1,\theta_2 \in {\mathbb{R}}$ such that at least one of $f_{\vec{\sigma}, a, \vec{z}} (\vec{0}) (\exp(e^{\theta_j})-z_2) (\exp(e^{\eta_j} + e^{\theta_j})-z_1)$, where $j=1,2$, are not real numbers. Thus 
$$
\frac{e^{\sigma_1 \eta} e^{\sigma_2 \theta} \exp((1-a)(e^\eta + e^\theta))}
{f_{\vec{\sigma},a,\vec{z}} (\vec{0}) (\exp(e^\theta)-z_2) (\exp(e^\eta + e^\theta)-z_1)} d\eta d\theta
$$
is not a measure but a complex signed measure. Hence $F_{\vec{\sigma}, a, \vec{z}} (\vec{t})$ is not a characteristic function in this case. 
\end{proof}

\begin{proof}[Proof of (2), (3) and (4) of Theorem \ref{th:dz2}]
Consider the case when $z_1=1$ and $z_2 \ne 1$. Note that the integral (\ref{eq:incase2}) converges absolutely when $\sigma_1>1$ and $\sigma_2 >0$. Hence we have that $F_{\vec{\sigma}, a, \vec{z}} (\vec{t})$ is a characteristic function when $z_1=1$ and $0 \ne z_2 \in [-1,1)$ by using (2) of Lemma \ref{lem:posi} and modifying the proof of Theorem \ref{th:dz1}. When $z_1=1$ and $z_2 \not \in [-1,1)$, we can see that $F_{\vec{\sigma}, a, \vec{z}} (\vec{t})$ is not a characteristic function from the manner used in the proof of Theorem \ref{th:dz1} since we can find $\eta_0, \eta_2, \theta_1,\theta_2 \in {\mathbb{R}}$ such that at least one of $f_{\vec{\sigma}, a, \vec{z}} (\vec{0}) (\exp(e^{\theta_j})-z_2) (\exp(e^{\eta_j} + e^{\theta_j})-1)$, where $j=1,2$, are not real. Similarly, we can show the cases (3) and (4).
\end{proof}

In order to show (1) of Theorem \ref{th:dz2}, we quote the following lemmas. 
\begin{lemma}[{see \cite[(2.16)]{NaHuLe}}] \label{lem:ezh2}
For $0< \Re (s_1) <1$, $\Re (s_2) >1$ and $1<\Re (s_1+s_2)<2$, we have the integral representation
\begin{equation}
\label{eq:gamhur*}
\begin{split}
&\Gamma (s_1) \Gamma (s_2) \zeta_2 (s_1,s_2 \,; a) = \\
&\int_0^\infty \!\!\! \int_0^\infty \frac{y^{s_2-1}}{e^y-1} H(a,x+y)x^{s_1-1} dx dy +
\int_0^\infty \!\!\! \int_0^\infty \frac{x^{s_1-1}}{x+y} H(1,y)y^{s_2-1} dx dy,
\end{split}
\end{equation}
where $H(a,x)$ is defined by (\ref{eq:defHax}).
\end{lemma}
\begin{lemma}[{see \cite[(2.20)]{NaHuLe}}] \label{lem:HuEZnega}
When $a \ge 1/2$, one has $\zeta_2 (\sigma_1,\sigma_2\,;a) < 0$ for $0<\sigma_1<1$, $\sigma_2>1$ and $1<\sigma_1+\sigma_2<2$. 
\end{lemma}

Next we prove the following lemma to prove (1) of Theorem \ref{th:dz2}.
\begin{lemma}\label{lem:negdefi2}
The function
$$
{\mathcal{H}}(a\,;x,y) := \frac{H(a,x+y)}{e^y-1} + \frac{H(1,y)}{x+y} 
$$
is negative for all $x,y>0$ if and only if $a \ge 1/2$.
\end{lemma}
\begin{proof}
First suppose $a \ge 1/2$. Then we obtain ${\mathcal{H}}(a\,;x,y)<0$ by using Lemma \ref{lem:negdefi}.

Secondly, let $0<a<1/2$ and $N \in {\mathbb{N}}$. Then we have
$$
\lim_{y\to +0} y{\mathcal{H}}(a\,;Ny,y) = \frac{1}{2} - a + \frac{1}{N+1} \times \biggl( \frac{-1}{2}\biggr) 
$$
by using the following formula (see also \cite[(2.4)]{NaHuLe})
$$
H(a,x) = \frac{(1/2-a)x^2 + ((1-a)^2/2!-1/3!)x^3+ \cdots}{x^2+ x^3/2! + \cdots} .
$$
Hence we obtain ${\mathcal{H}}(a\,;Nx,y)>0$ when $y$ and $x=Ny$ are sufficiently small and $N$ is sufficiently large (for example, take $N=[y^{-1/2}]$). Note that 
$$
h(a,x):=x(e^x-1)H(a,x) = xe^{(1-a)x} - e^x +1<0
$$ 
when $x$ is sufficiently large since one has $e^{1-a} < e$. Thus ${\mathcal{H}}(a\,;x,y)<0$ when $x$ and $y$ are sufficiently large.
\end{proof}

\begin{proof}[Proof of (1) of Theorem \ref{th:dz2}]
By the change of variables integration $x=e^\eta$ and $y=e^\theta$ in (\ref{eq:gamhur*}), it holds that
\begin{equation*}
\begin{split}
&f_{\vec{\sigma}, a, \vec{1}} (\vec{0}) F_{\vec{\sigma}, a, \vec{1}} (\vec{t}) \\
&= \int_{-\infty}^{\infty} \int_{-\infty}^{\infty} \left(
\frac{H(a,e^\eta+e^\theta)}{\exp(e^\theta)-1} + \frac{H(1,e^\theta)}{e^\eta + e^\theta} \right)
e^{\eta(\sigma_1+{\rm{i}}t_1-1)} e^{\theta(\sigma_2+{\rm{i}}t_2-1)} e^\eta e^\theta d\eta d\theta \\ &=
\int_{-\infty}^{\infty} \int_{-\infty}^{\infty} e^{{\rm{i}}( t_1\eta + t_2\theta)} \left(
\frac{H(a,e^\eta+e^\theta)}{\exp(e^\theta)-1} + \frac{H(1,e^\theta)}{e^\eta + e^\theta} \right)
e^{\sigma_1 \eta} e^{\sigma_2 \theta} d\eta d\theta .
\end{split}
\end{equation*}

Suppose $a \ge 1/2$. Then we have $H(a,e^\eta+e^\theta)(\exp(e^\theta)-1)^{-1} + H(1,e^\theta) (e^\eta + e^\theta)^{-1} < 0$ for any $\eta,\theta \in {\mathbb{R}}$ by Lemma \ref{lem:negdefi2}. Furthermore, one has $\zeta (\sigma_1,\sigma_2 \,;a)<0$ for all $0<\sigma_1<1$, $\sigma_2>1$ and $1<\sigma_1+\sigma_2<2$ by Lemma \ref{lem:HuEZnega}. Obviously we have $\int_{\mathbb{R}^2} P_{\vec{\sigma}, a, \vec{1}} (\eta, \theta) d\eta d\theta =1$ for any $0<\sigma_1<1$, $\sigma_2>1$ and $1<\sigma_1+\sigma_2<2$ from the definition of $P_{\vec{\sigma}, a, \vec{1}} (\eta, \theta)$ given by (\ref{eq:qdef2}). Therefore, $P_{\vec{\sigma}, a, \vec{1}} (\eta, \theta)$ is a probability density function when $a \ge 1/2$. 

Next suppose $0 <a <1/2$. Then there exist $\eta_1,\eta_2, \theta_1, \theta_2 \in {\mathbb{R}}$ such that 
$$
\frac{H(a,e^{\eta_1}+e^{\theta_1})}{\exp(e^{\theta_1})-1} + 
\frac{H(1,e^\theta_1)}{e^{\eta_1} + e^{\theta_1}} >0,  \qquad
\frac{H(a,e^{\eta_2}+e^{\theta_2})}{\exp(e^{\theta_2})-1} + 
\frac{H(1,e^{\theta_2})}{e^{\eta_2} + e^{\theta_2}} < 0
$$
by Lemma \ref{lem:negdefi2}. Hence 
\begin{equation}\label{mea:2}
\frac{e^{\sigma_1 \eta} e^{\sigma_2 \theta}}{f_{\vec{\sigma}, a} (\vec{0})} \left(
\frac{H(a,e^\eta+e^\theta)}{\exp(e^\theta)-1} + \frac{H(1,e^\theta)}{e^\eta + e^\theta} \right) d\eta d\theta 
\end{equation}
is not a measure but signed measure. From the view of 
\begin{equation*}
\begin{split}
&\int_1^\infty \!\!\! \int_1^\infty \frac{y^{\sigma_2-1}}{e^y-1} |H(a,x+y)| x^{\sigma_1-1} dx dy  < \infty ,\\
&\int_0^\infty \!\!\! \int_0^\infty \frac{x^{\sigma_1-1}}{x+y} |H(1,y)| y^{\sigma_2-1} dx dy < \infty
\end{split}
\end{equation*}
(see the proof of \cite[Proposition 2.6]{NaHuLe}), the signed measure (\ref{mea:2}) has finite total variation. Therefore, $F_{\vec{\sigma}, a, \vec{1}} (\vec{t})$ is not a characteristic function when $0 <a <1/2$ since any signed measure with finite total variation is uniquely determined by the Fourier transform. 
\end{proof}

\begin{remark}\label{re:las}
Let $1/2 < \sigma_0 <1$ such that $\zeta_2 (\sigma_0, \sigma_0 \,; a)=0$. Then
$$
\frac{\Gamma (\sigma_0 + {\rm {i}} t)^2}{\Gamma (\sigma_0)^2} 
\frac{\zeta_2 (\sigma_0 + {\rm {i}} t, \sigma_0 + {\rm {i}} t \,; a)}{\zeta_2 (\sigma_0, \sigma_0 \,; a)}
$$
is not a characteristic function since one has $0<\Gamma (\sigma_0) <\infty$ and $\zeta_2 (\sigma_0, \sigma_0 \,; a)=0$ (see \cite[Proposition 1.5]{NaHuLe}). Thus we need the non-existence of real zeros of zeta functions in oder to define zeta distributions.
\end{remark}

%%%%%%%%%%%%%%%%%%%%%%%%%%
 
\end{document}